\documentclass[12pt, leqno]{amsart}

\usepackage{amsfonts,amsthm,amsmath,xcolor}
\numberwithin{equation}{section}

\usepackage{amssymb}

\theoremstyle{plain}
\newtheorem{thm}{Theorem}[section]

\newtheorem{lem}{Lemma}[section]

\theoremstyle{definition}
\newtheorem{df}{Definition}[section]
\newtheorem{rem}{Remark}[section]
\newtheorem{ex}{Example}[section]

\newcommand{\FF}{\mathbb{F}}
\newcommand{\ZZ}{\mathbb{Z}}

\newcommand{\CC}{\mathbb{C}}

\newcommand{\I}{\mathrm{I}}
\newcommand{\II}{\mathrm{II}}
\newcommand{\III}{\mathrm{III}}
\newcommand{\IV}{\mathrm{IV}}

\def\bm#1{\mathbf{#1}}

\DeclareMathOperator{\supp}{supp}

\DeclareMathOperator{\wt}{wt}

\begin{document}


\title{Jacobi polynomials and design theory I}

\author[Chakraborty]{Himadri Shekhar Chakraborty}
\address{Department of Mathematics, Shahjalal University of Science and Technology, Sylhet 3114, BANGLADESH}
\email{himadri-mat@sust.edu}
 
\author[Miezaki]{Tsuyoshi Miezaki}
\address{School of Fundamental Science and Engineering, Waseda University, Tokyo 169-8555, Japan}
\email{miezaki@waseda.jp}

\author[Oura]{Manabu Oura}
\address{School of Natural Science and Technology, Kanazawa University, Ishikawa 920-1192,Japan}
\email{oura@se.kanazawa-u.ac.jp}

\author[Tanaka]{Yuuho Tanaka$^\ast$}
\address{Graduate School of Science and Engineering, Waseda University, Tokyo, 169-8555, Japan}
\email{tanaka\_yuuho\_dc@akane.waseda.jp}

\thanks {*Corresponding author}

\date{}
\maketitle

\begin{abstract}
In this paper, we introduce the notion of Jacobi polynomials of a code with multiple 
reference vectors, and give the MacWilliams type identity for it.
Moreover, we derive a formula to obtain the Jacobi polynomials using the 
Aronhold polarization operator.
Finally, we describe some facts obtained from Type III and Type IV codes
that interpret the relation between the Jacobi polynomials and designs.

\end{abstract}

{\small
	\vspace{-0.15in}
	
	\quad
}

\section{Introduction}

A.~Bonnecaze et al.~\cite{BMS1999} took the notion of 
Jacobi polynomials, a celebrated generalization of 
weight enumerators~\cite{HP, MS1977} that were 
introduced by M.~Ozeki~\cite{Ozeki} for codes as an analogue
to Jacobi forms~\cite{BO,EZ} as a powerful generalization of 
modular form~\cite{Duke, Runge} of Lattices~\cite{CS}. 
They gave a formula to compute the Jacobi polynomials of a
binary code as an application of combinatorial $t$-designs using an operator,
known as Aronhold polarization operator. 
Many authors studied the combinatorial $t$-designs and discussed 
their properties in~\cite{AM1969, Delsarte, Miezaki2021, MMN2021}
that were derived from codes and their analogies. 
Moreover, P.J.~Cameron~\cite{Cameron2009} gave the notion of generalized $t$-designs
and discussed its properties. 
Furthermore, A.~Bonnecaze et al.~\cite{BMS1999}
constructed various types of designs such as 
group divisible designs, packing designs and covering designs.
To establish the relationship between these designs and the Jacobi polynomials,
they studied Jacobi polynomials for Type~II codes through 
invariant theory~\cite{Gleason, NRS}.

In this paper, 
we give the generalizations and analogues of some results in~\cite{BMS1999}.
We define the Jacobi polynomials with multiple reference vectors 
for codes, and give the MacWilliams type identity for it.
As an analogue of the combinatorial interpretation of the polarization that
was given in~\cite{BMS1999}, is given here for codes that holds generalized $t$-designs
for every given weight of the codewords.
In addition, 
we study some Type~III (resp. Type~IV) codes of specific lengths,
and determine the polynomials that generate the space of Jacobi polynomials for 
a Type~III (resp. Type~IV) code
with respect to reference vectors of a particular length. 
Moreover, we observe from the examples that
the number of blocks of a packing (resp. covering) design 
correspond to the coefficients in Jacobi polynomials. 

This paper is organized as follows. 
In Section~\ref{Sec:Preli}, 
we discuss the basic definitions and properties of codes that needed to understand this paper. 
In Section~\ref{Sec:MacWilliams}, 
we give the MacWilliams type identity (Theorem~\ref{Thm:MacWilliams}) for 
the Jacobi polynomials of a code with multiple reference vectors.
In Section~\ref{Sec:DesignJac}, 
we see how polarization operator acts to obtain the Jacobi polynomials with multiple reference vectors (Theorem~\ref{Thm:Main1}, Theorem~\ref{Thm:Main2}).
In Section~\ref{Sec:Examples},
we disclose some facts between a Type~III (resp. Type~IV) code of specific length 
and designs of various kinds with the help of the Jacobi polynomials.
Finally, we conclude the paper with some remarks in Section~\ref{Sec:Conclusion}. 
%

All computer calculations in this paper were done with the help of
Magma~\cite{Magma}.

\section{Preliminaries}\label{Sec:Preli}

Let $\FF_{q}$ be a finite field of order~$q$, 
where $q$ is a prime power. 
Then $\FF_{q}^{n}$ denotes the vector space of dimension $n$ over $\FF_{q}$.
The elements of $\FF_{q}^{n}$ are known as \emph{vectors}.
The \emph{Hamming weight} of 
$\bm{u} = (u_1,\dots, u_n)\in \FF_{q}^{n}$ 
is denoted by~$\wt(\bm{u})$ and
defined to be the number of $i$'s such that $u_{i} \neq 0$.
Let $\bm{u} = (u_1,\dots, u_n)$ 
and $\bm{v} = (v_1,\dots,v_n)$ 
be the vectors of $\FF_{q}^{n}$. 
Then the \emph{inner product} of two vectors 
$\bm{u},\bm{v} \in \FF_{q}^{n}$ 
is given by
\[
	\bm{u} \cdot \bm{v} 
	:= 
	u_{1} v_{1} + \dots + u_{n}v_{n}.
\]
If $q$ is an even power of an arbitrary prime~$p$,
then it is convenient to consider another inner product
given by
\[
	\bm{u} \cdot \bm{v}
	:=
	u_{1}\overline{v_{1}} + \cdots + u_{n}\overline{v_{n}},
\]
where $\overline{v_{i}}:= {v_{i}}^{\sqrt{q}}$.
An $\FF_q$-linear code of length~$n$ is a vector subspace of $\FF_{q}^{n}$. 
The elements of an~$\FF_{q}$-linear code are called \emph{codewords}. 
The \emph{dual code} of an $\FF_{q}$-linear code~$C$ 
of length~$n$ is defined by
\[
	C^\perp 
	:= 
	\{
		\bm{v}\in \FF_{q}^{n} 
		\mid 
		\bm{u} \cdot \bm{v} = \bm{0} 
		\text{ for all } 
		\bm{u}\in C
	\}. 
\]
An $\FF_{q}$-linear code $C$ is called \emph{self-dual} if $C = C^\perp$. 
It is well known that 
the length~$n$ of a self-dual code over~$\FF_q$ is even and the dimension is $n/2$.
To study self-dual codes in detail, we refer the readers 
to~\cite{BMS1972, Gleason, MMS1972, NRS}. 
A self-dual code~$C$ over~$\FF_2$ or~$\FF_4$ of length $n\equiv 0 \pmod 2$ having even weight is called \emph{Type}~$\I$ and \emph{Type}~$\IV$, respectively. A self-dual code $C$ over~$\FF_2$ of length $n\equiv 0\pmod 8$ is called \emph{Type}~$\II$ if the weight of each codeword of~$C$ is multiple of~$4$. Finally, a self-dual code $C$ over~$\FF_3$ of length $n\equiv 0\pmod 4$ is called \emph{Type}~$\III$ if the weight of each codeword of~$C$ is multiple of~$3$.


\begin{df}
	Let $C$ be an $\FF_{q}$-linear code of length~$n$.
	We denote by $A_{i}^{C}$ the number of codewords in $C$ having 
	Hamming weight $i$.
	Then the \emph{weight enumerator} of $C$ 
	is defined as
	\[
		W_{C}(x,y)
		:=
		\sum_{\bm{u} \in C}
		x^{n-\wt(\bm{u})}
		y^{\wt(\bm{u})}
		=
		\sum_{i=0}^{n}
		A_{i}^{C}
		x^{n-i}
		y^{i}.
	\]
\end{df}

\begin{df}\label{Def:JacOne}
	Let $C$ be an $\FF_{q}$-linear code of length~$n$. Then 
	the \emph{Jacobi polynomial} attached to a set $T$ of 
	coordinate places of the code $C$ is defined as follows:
	\[
		J_{C,T}(w,z,x,y) 
		:=
		\sum_{\bm{u}\in C}
		w^{m_0(\bm{u})}z^{m_1(\bm{u})}x^{n_0(\bm{u})}y^{n_1(\bm{u})},
	\]
	where $T\subseteq [n]$, and for $\bm{u}\in C$,
	\begin{align*}
	m_0(\bm{u}) & := \#\{i\in T \mid u_i=0\},\\
	m_{1}(\bm{u}) & := \#\{i \in T \mid u_{i} \neq 0\},\\
	n_0(\bm{u}) & := \#\{i\in[n]\setminus T \mid u_i= 0\},\\
	n_{1}(\bm{u})  & := \#\{i\in[n]\setminus T \mid u_i \neq 0\}.
	\end{align*}
\end{df}

\begin{rem}
	If $T \subseteq [n]$ is empty, then $J_{C,T}(w,z,x,y)=W_C(x,y)$.
\end{rem}

\section{MacWilliams type identity}\label{Sec:MacWilliams}

The MacWilliams type identity for the Jacobi polynomial of 
an $\FF_{q}$--linear code with one reference vector was given in~\cite{Ozeki}. 
In this section, we give the MacWilliams type identity for
the Jacobi polynomial of an $\FF_{q}$--linear code with multiple reference
vectors.

\begin{df}
	Let $C$ be an $\FF_{q}$--linear code of length~$n$. 
	Then the \emph{Jacobi polynomial} of $C$ 
	with respect to $\ell$ reference vectors 
	$\bm{w}_{1},\ldots,\bm{w}_{\ell} \in \FF_{q}^{n}$ 
	is denoted by
	$J_{C,\bm{w}_{1},\ldots,\bm{w}_{\ell}}(\{x_{\bm{a}}\}_{\bm{a} \in \FF_{2}^{\ell+1}})$ and 	
	defined as
	\begin{align*}
		J_{C,\bm{w}_{1},\ldots,\bm{w}_{\ell}}
		(\{x_{\bm{a}}\}_{\bm{a} \in \FF_{2}^{\ell+1}})
		& := 
		\sum_{\bm{u}\in C} 
		\prod_{\bm{a} \in \FF_{2}^{\ell+1}}
		x_{\bm{a}}^{N_{\bm{a}}(\bm{u},\bm{w}_{1},\ldots,\bm{w}_{\ell})}.
	\end{align*}
	Here 
	for $\bm{a} \in \FF_{2}^{\ell}$,
	we denote
	$$N_{\bm{a}}(\bm{w}_{1},\ldots,\bm{w}_{\ell}) := \#\{i \in [n] \mid \bm{a} = (\phi(w_{1i}),\ldots,\phi(w_{gi}))\},$$ 
	where $\bm{w}_{j} = (w_{j1},\ldots,w_{jn})$, 
	$\phi(w_{ji}) = 1$ if $w_{ji} \neq 0$ and $\phi(w_{ji}) = 0$
	if~$w_{ji} = 0$.
\end{df}

Note that if $\ell=1$, 
the above definition is completely equivalent to the Jacobi polynomial with one 
reference vector (Definition~\ref{Def:JacOne}). 

%
%

Let $\FF_{q}$ be a finite field, where $q =  p^{f}$ for some prime number $p$. 
A \emph{character} of $\FF_{q}$ is a homomorphism from the additive group~$\FF_{q}$ 
to the multiplicative group of non-zero complex numbers. 
We review~\cite{CM2021,MS1977} to introduce some fixed non-trivial characters over~$\FF_{q}$. 
Now let $F(x)$ be a primitive irreducible polynomial 
of degree $f$ over $\FF_{p}$ and let $\lambda$ be a root of $F(x)$.
Then any element $a \in \FF_{q}$ has a unique representation as:
\begin{equation*}\label{EquAlpha}
	a 
	=
	a_{0} + a_{1} \lambda	
	+ a_{2} \lambda^{2}
	+ \cdots +
	a_{f-1} \lambda^{f-1},
\end{equation*} 
where $a_{i} \in \FF_{p}$.
For $b \in \FF_{q}$, 
we define $\chi_{b}(a) := \zeta_{p}^{a_{0}b_{0}+\cdots+a_{f-1}b_{f-1}}$, 
where $\zeta_{p}$ is the $p$-th primitive root $e^{2{\pi}i/p}$ of unity.
When $b \neq 0$, then $\chi_{b}$ is a non-trivial character of $\FF_{q}$.
Let $\chi$ be a non-trivial character of $\FF_{q}$. 
Then for any $a \in \FF_{q}$, 
we have the following property:
\[
	\sum_{b \in \FF_{q}}
	\chi(ab) 
	:= 
	\begin{cases}
		q & \mbox{if} \quad a = 0, \\
		0 & \mbox{if} \quad a \neq 0.
	\end{cases} 
\]

\begin{lem}[\cite{MS1977}]\label{Lem:DualIden}
	Let $C$ be an $\FF_{q}$-linear code of length~$n$. 
	For $\bm{v} \in \FF_{q}^{n}$, define
	\[
		\delta_{C^{\perp}}(\bm{v}) 
		:= 
		\begin{cases}
			1 & \mbox{if } \bm{v} \in C^{\perp}, \\
			0 & \mbox{otherwise}.
	\end{cases} 
	\]
	Then we have the following identity:
	\[
		\delta_{C^\perp}(\bm{v})
		= 
		\dfrac{1}{|C|} 
		\sum_{\bm{u} \in C} 
		\chi(\bm{u} \cdot \bm{v}).
	\]
\end{lem}

Now we give the MacWilliams type identity
for the Jacobi polynomial of an $\FF_{q}$--linear code 
with respect to multiple reference vectors.

\begin{thm}[MacWilliams Identity]\label{Thm:MacWilliams}
	Let $C$ be an $\FF_{q}$--linear code of length $n$.
	Again let $\chi$ be a non-trivial character of $\FF_{q}$.
	Let 
	$J_{C,\bm{w}_{1},\ldots,\bm{w}_{\ell}} 
	(\{x_{\bm{a}}\}_{\bm{a} \in \FF_{2}^{\ell+1}})$ 
	be the Jacobi polynomial of $C$ 
	with respect to the reference vectors 
	$\bm{w}_{1},\ldots,\bm{w}_{\ell}\in \FF_{q}^{n}$. 
	Then
	\begin{align*}
		J&_{C^{\perp},\bm{w}_{1},\ldots,\bm{w}_{\ell}} 
		(\{x_{\bm{a}}\}_{\bm{a} \in \FF_{2}^{\ell+1}})\\
		& = 
		\dfrac{1}{|C|} 
		J_{C,\bm{w}_{1},\ldots,\bm{w}_{\ell}}
		\left(
			\left\{
				\sum_{b \in \FF_{q}}
				\chi(a_{1} b) 
				x_{(\phi(b), \phi(a_{2}), \ldots, \phi(a_{\ell+1}))}
			\right\}_{\bm{a} \in \FF_{q}^{\ell+1}}
		\right).
	\end{align*}
\end{thm}

\begin{proof}
	By Lemma~\ref{Lem:DualIden}, we can write
	
	\begin{align*}
		J&_{C^{\perp},\bm{w}_{1},\ldots,\bm{w}_{\ell}}
		(\{x_{\bm{a}}\}_{\bm{a} \in \FF_{2}^{\ell+1}}) \\
		& = 
		\sum_{\bm{u} \in C^{\perp}} 
		\prod_{\bm{a} \in \FF_{2}^{\ell+1}} 
		x_{\bm{a}}
		^{N_{\bm{a}}(\bm{u},\bm{w}_{1},\ldots,\bm{w}_{\ell})} \\
		& = 
		\sum_{\bm{v} \in \FF_{q}^{n}} 
		\delta_{C^\perp}(\bm{v}) 
		\prod_{\bm{a} \in \FF_{2}^{\ell+1}} 
		x_{\bm{a}}
		^{N_{\bm{a}}(\bm{v},\bm{w}_{1},\ldots,\bm{w}_{\ell})} \\
		& =
		\dfrac{1}{|C|} 
		\sum_{\substack{\bm{u} \in C\\ 
			\bm{v}\in \FF_{q}^{n}}}  
		\chi(\bm{u} \cdot \bm{v}) 
		\prod_{\bm{a} \in \FF_{2}^{\ell+1}} 
		x_{\bm{a}}
		^{N_{\bm{a}}(\bm{v},\bm{w}_{1},\ldots,\bm{w}_{\ell})}\\
		& = 
		\dfrac{1}{|C|}
		\sum_{\substack{\bm{u} \in C\\ 
				\bm{v} \in \FF_{q}^{n}}} 
		\chi(u_{1}v_{1} + \cdots + u_{n}v_{n})
		\prod_{1\leq i \leq n} 
		x_{(\phi(v_{i}), \phi(w_{1i}), \ldots, \phi(w_{\ell i}))} \\
		& = 
		\dfrac{1}{|C|}
		\sum_{\bm{u} \in C}
		\prod_{1\leq i \leq n} 
		\left\{
			\sum_{v_{i} \in \FF_{q}} 
			\chi(u_{i}v_{i}) 
			x_{(\phi(v_{i}), \phi(w_{1i}), \ldots, \phi(w_{\ell i}))} 
		\right\}\\
		& = 
		\dfrac{1}{|C|}
		\sum_{\bm{u} \in C}
		\prod_{\bm{a} \in \FF_{q}^{\ell+1}} 
		\left(
			\sum_{b \in \FF_{q}}
			\chi(a_{1} b) 
			x_{(\phi(b), \phi(a_{2}), \ldots, \phi(a_{\ell+1}))}
		\right)
		^{N_{\bm{a}}(\bm{u},\bm{w}_{1},\ldots,\bm{w}_{\ell})} \\
		& = 
		\dfrac{1}{|C|}
		J_{C,\bm{w}_{1},\ldots,\bm{w}_{\ell}}
		\left(
			\left\{
				\sum_{b \in \FF_{q}}
				\chi(a_{1} b) 
				x_{(\phi(b), \phi(a_{2}), \ldots, \phi(a_{\ell}+1))}
			\right\}_{\bm{a} \in \FF_{q}^{\ell+1}}
		\right).	
	\end{align*}
	Hence the proof is completed.
\end{proof}

\section{Generalized $t$-designs and Jacobi polynomials}\label{Sec:DesignJac}

Bonnecaze et al. introduced an operator called \emph{polarization operator} in~\cite{BMS1999}, 
and using this operator, they gave a formula to evaluate the Jacobi polynomial of a binary code from the weight enumerator of the code. 
In this section, we give a generalized form of the polarization operation, 
and present an application of this operator in the evaluation of the Jacobi polynomial
of a non-binary code associated to the multiple reference vectors.

First we recall the definition of generalized $t$-designs 
from~\cite{Cameron2009}.
Let $t$, $k$, $\lambda$ be the integers such that $\lambda > 0$ and $k > t > 0$. 
Again let $\bm{k} := (k_{1},\ldots,k_{n})$ such that $k = \sum_{i = 1}^{n} k_{i}$,
$\bm{v} := (v_{1},\ldots,v_{n})$ such that $v_{i} \geq k_{i}$ for all $i$.
Let $\bm{X} := (X_{1},\ldots,X_{n})$, where $X_{i}$'s are pairwise disjoint sets with $|X_{i}| = v_{i}$ for all $i$ and  
$$\mathcal{B} \subseteq \binom{X_{1}}{k_{1}} \times \cdots \times \binom{X_{n}}{k_{n}}.$$

\begin{df}	
	A $t$-$(\bm{v},\bm{k},\lambda)$ \emph{design} or a \emph{generalized} $t$-\emph{design} (in short) is a pair $\mathcal{D} := (\bm{X},\mathcal{B})$ 
	with the following property:
	if $\bm{t} := (t_{1},\ldots,t_{n})$ such that $t = \sum_{i = 1}^{n} t_{i}$ satisfying $0 \leq t_{i} \leq k_{i}$ for all $i$, 
	then for any choice $\bm{T} := (T_{1},\ldots,T_{n})$ with $T_{i} \in \binom{X_{i}}{t_{i}}$ for all $i$, 
	there are precisely $\lambda$ members $\bm{K} := (K_{1},\ldots,K_{n}) \in \mathcal{B}$ for which $T_{i} \subseteq K_{i}$ for all $i$.
\end{df}

Note that in the case when $\bm{k} = (k)$ and $\bm{v} = (v)$, 
this is precisely the definition of a combinatorial 
$t$-$(v,k,\lambda)$ design or a $t$-\emph{design} (in short). 
We can construct the generalized $t$-designs from codes as follows.

Let $\bm{v} = (v_{1},\ldots,v_{\ell})$ such that $\sum_{i = 1}^{\ell} v_{i} = n$ and
$\bm{X} = (X_{1},\ldots,X_{\ell})$ of pairwise disjoint sets $X_{i} \subseteq [n]$ with $|X_{i}| = v_{i}$.
Again let $\bm{u} = (u_{1},\ldots,u_{n})\in \FF_{q}^{n}$. 
Then for $X \subseteq [n]$, we define
\begin{align*}
	\supp_{X}(\bm{u}) &:= \{i \in X \mid u_{i} \neq 0\},\\
	\bm{K}(\bm{u}) &:= (\supp_{X_{1}}(\bm{u}),\ldots,\supp_{X_{\ell}}(\bm{u})),\\
	\wt_{X}(\bm{u}) &:= |\supp_{X}(\bm{u})|.
\end{align*}

Again for any positive integer $k$, let 
$\bm{k} = (k_{1},\ldots,k_{\ell})$ such that 
$\sum_{i = 1}^{\ell} k_{i} = k$.
Let $C$ be an $\FF_{q}$--linear code of length~$n$. Then
\begin{align*}
	C_{\bm{k}} 
	& := 
	\{\bm{u} \in C \mid \wt_{X_{i}}(\bm{u}) = k_{i} \mbox{ for all } i\},\\  \mathcal{B}(C_{\bm{k}}) 
	& := 
	\{\bm{K}(\bm{u}) \mid \bm{u} \in C_{\bm{k}}\}.
\end{align*}
In general, $\mathcal{B}(C_{\bm{k}})$ is a multi-set. We say $C_{\bm{k}}$ is a $t$-$(\bm{v},\bm{k},\lambda)$ design if $(\bm{X},\mathcal{B}(C_{\bm{k}}))$ is a $t$-$(\bm{v},\bm{k},\lambda)$ design. We say a code is \emph{generalized} $t$-\emph{homogeneous} if the codewords of every given weight~$k$ hold a 
$t$-$(\bm{v},\bm{k},\lambda)$ design.

From the above discussion we have the following result. 
We omit the proof of the theorem since it follows from the above definitions. 

\begin{thm}\label{Thm:JacToDesign}
	Let $C$ be an $\FF_{q}$-linear code of length~$n$.
	Let~$t, k, \lambda$ be the integers such that $\lambda > 0$ and $k > t > 0$. 
	Let $\bm{v} := (v_{1},\ldots,v_{\ell})$ such that $v_{1}+\cdots + v_{\ell} = n$.
	Let $\bm{X} := (X_{1},\ldots,X_{\ell})$ of pairwise disjoint set 
	$X_{1},\ldots, X_{\ell} \subseteq [n]$ 
	with $|X_{1}| = v_{1},\ldots,|X_{\ell}| = v_{\ell}$.
	Let $\bm{k} := (k_{1},\ldots,k_{\ell})$ such that $k_{1}+\cdots + k_{\ell} = k$.
	Then the set of codewords of $C$ form a 
	$t$-$(\bm{v},\bm{k},\lambda)$ design for every given weight $k$ with 
	$\bm{t} = (t_{1},\ldots,t_{\ell})$ such $t_{1}+\cdots + t_{\ell} = t$
	satisfying $0\leq t_{i} \leq k_{i}$ for all $i$
	if and only if the Jacobi polynomial
	$J_{C,\bm{w}_{1},\ldots,\bm{w}_{\ell}}$ 
	of $C$ associated to the reference vectors 
	$\bm{w}_{1},\ldots,\bm{w}_{\ell} \in \FF_{q}^{n}$
	such that $\wt_{X_{i}}(\bm{w}_{i}) = t_{i} = \wt(\bm{w}_{i})$ for all~$i$,
	is invariant.
\end{thm}

In the situation described in the above theorem,
the Jacobi polynomial $J_{C,\bm{w}_{1},\ldots,\bm{w}_{\ell}}$ is independent of
the choices of the reference vectors $\bm{w}_{1},\ldots,\bm{w}_{\ell} \in \FF_{q}^{n}$.
In this case, we prefer to denote the Jacobi polynomial as $J_{C,t_{1},\ldots,t_{\ell}}$.
In particular, when $\bm{k} = (k)$ and $\bm{t} = (t)$, it becomes the Jacobi polynomial
$J_{C,t}$ as in~\cite{BMS1999}.

Let $C$ be an $\FF_{q}$-linear code of length~$n$.
Then the code $C-i$ obtained from~$C$ by \emph{puncturing} at coordinate place $i$.
Now from~\cite{MS1977} we have the following lemma.

\begin{lem}\label{Lem:MacSloane}
	Let $C$ be a code of length~$n$. Then
	\[
	W_{C-i}(x,y) 
	= 
	\frac{1}{n}
	\left(	
	\frac{\partial }{\partial x} 
	+
	\frac{\partial}{\partial y}
	\right)
	W_{C}(x,y).
	\]
\end{lem}

Let $P(x_{0},x_{1})$ be a homogeneous polynomial with indeterminates 
$x_{0}$ and $x_{1}$. 
Again let $P'_{x_{0}}$(resp. $P'_{x_{1}}$) denote the partial derivative with respect to variable $x_{0}$ (resp. $x_{1}$). Define the polarization operator 
$A_{j}$ for any integer $1 \leq j \leq \ell$
as follows:
\begin{equation}\label{Equ:Operator}
	A_{j}
	\cdot
	 P(w_{j},z_{j},x_{0},x_{1}) 
	:= 
	w_{j}P'_{x_{0}}(x_{0},x_{1}) 
	+ 
	z_{j}P'_{x_{1}}(x_{0},x_{1}).
\end{equation}
Here the indeterminates in the above equation denote $x_{a}$
for some $a \in \FF_{2}^{\ell+1}$ as follows:
$w_{j} := x_{(0,0,\ldots,0,\underset{(j+1)\text{th}}{1},0,\ldots,0)}$,
$z_{j} := x_{(1,0,\ldots,0,\underset{(j+1)\text{th}}{1},0,\ldots,0)}$,
$x_{0} := x_{(0,0,\ldots,0)}$, and
$x_{1} := x_{(1,0,\ldots,0)}$.
Now we have a generalization of~\cite[Theorem 3]{BMS1999} as follows.


\begin{thm}\label{Thm:Main1}
	Every code $C$ is a generalized $1$-homogenous if and only if 
	for any $\ell$-tuple $(0,\ldots,0,1,0,\ldots,0)$ 
	having a single non-zero coordinate, 
	say $j$-th coordinate	with $1$, we have
	\begin{equation}\label{Equ:OneHomo}
		J_{C,0,\ldots,0,1,0,\ldots,0}
		=
		\frac{1}{n}
		A_{j} 
		\cdot 
		W_C.
	\end{equation}
\end{thm}

\begin{proof}
	Let $C$ be generalized $1$-homogeneous. 
	Then by Lemma~\ref{Lem:MacSloane} one can easily find 
	Equation~(\ref{Equ:OneHomo}) is true.
	Conversely, the hypothesis implies that the
	Jacobi polynomial 
	$J_{C,0,\ldots,0,1,0,\ldots,0}$
	is uniquely determined.
	Therefore, by Theorem~\ref{Thm:JacToDesign} we can say that the codewords of every given weight of $C$ form a generalized $1$-design.  
	Hence~$C$ is generalized $1$-homogeneous.
\end{proof}

\begin{thm}\label{Thm:Main2}
	If $C$ is generalized $t$-homogeneous and contains no codeword of weight $\leq t$ then for $\bm{t} = (t_{1},\ldots,t_{\ell})$ such that $t_{1} + \cdots + t_{\ell} = t$ we have
	\[
		J_{C,t_{1},\ldots,t_{\ell}}
		=
		\frac{1}{n(n-1)\cdots(n-t+1)}
		A_{\ell}^{t_{\ell}} \cdots A_{1}^{t_{1}} \cdot W_C.
	\]
\end{thm}

\begin{proof}
	The statement is true for $t = 1$ by Theorem~\ref{Thm:Main1}. 
	For $\mathbf{d} := (d_1,\ldots,d_{\ell})$ such that 
	$d_1 + \cdots + d_\ell = d < t$ satisfying $0 \leq d_{i} \leq t_{i}$
	for all $i$, let us suppose that
	\[
		J_{C,d_{1},\ldots,d_{\ell}}
		=
		\frac{1}{n(n-1)\cdots(n-d+1)}
		A_{\ell}^{d_{\ell}} \cdots A_{1}^{d_{1}} \cdot W_C.
	\]
	Let $d_{j} < t_{j}$. Then we have 
	
	\begin{align*}
		J_{C,d_{1},\ldots,d_{j-1},(d_{j}+1),d_{j+1},\ldots,d_{\ell}}
		& =
		\frac{1}{n-d}
		A_{j} \cdot J_{C,d_{1},\ldots,d_{j-1},d_{j},d_{j+1},\ldots,d_{\ell}}\\
		& =
		\frac{1}{n-d}
		A_{j}
		\frac{1}{n(n-1)\cdots(n-d+1)}\\
		& \quad \quad \quad
		A_{\ell}^{d_{\ell}} \cdots A_{j+1}^{d_{j+1}}A_{j}^{d_{j}}A_{j-1}^{d_{j-1}}\cdots A_{1}^{d_{1}} 
		\cdot W_C\\
		& =
		\frac{1}{n(n-1)\cdots(n-d+1)(n-d)}\\
		& \quad \quad \quad
		A_{\ell}^{d_{\ell}} \cdots A_{j+1}^{d_{j+1}}A_{j}^{d_{j}+1}A_{j-1}^{d_{j-1}}\cdots A_{1}^{d_{1}} 
		\cdot W_C.
	\end{align*}
	The converse implication follows from the proof of Theorem~\ref{Thm:Main1}.
\end{proof}

\section{Designs and Molien series}\label{Sec:Examples}

Bonnecaze et al.~\cite{BMS1999} studied certain length of Type~II codes to
focus some relation between Jacobi polynomials and designs. 
In this section, we follow the idea, and establish the connection between Jacobi
polynomials and designs for some Type~III and Type~IV codes. We would like to mention
that in this section, we study Jacobi polynomials with one reference vector. 
%
%
To overcome all sorts of confusions, we refer the readers to~\cite{BMS1999}
for notations and symbols.

First, let us recall~\cite{BMS1999} for the definitions of various types of designs. 
A \emph{design} with parameters $t$-$(v,k,(\lambda_1^{a_1},\dots,\lambda_N^{a_N}))$ is a collection of $k$-element subsets called \emph{blocks} of a $v$-element set (the \emph{varieties}) and a partition of the set of all $t$-tuples into $N$ \emph{groups} such that every $t$-set belonging to the $i^{\text{th}}$ group (comprising $a_i$ such $t$-sets) is contained in exactly $\lambda_i$ blocks. 
Notice that for $N = 1$ the design coincide with a $t$-design.
A \emph{packing} (resp. \emph{covering}) \emph{design} with parameters $t$-$(v,k,\lambda)$ is a design with $\max_i(\lambda_i)=\lambda$
(resp. $\min_i(\lambda_i)=\lambda$).
The maximum (resp. minimum) number of blocks of a packing (resp. covering) design denoted by $D_{\lambda}(v,k,t)$ (resp. $C_{\lambda}(v,k,t)$).

The study of weight polynomials of a code with the help of invariant theory
is a very convenient and powerful technique, as shown in~\cite{NRS, Slone1977}.          
Let $G$ be a finite subgroup of $GL(2,\CC)$, and 
$G$ acts on a polynomial ring of two variables, say $x,y$.
Then it is well-known from~\cite[Theorem 1]{Slone1977} that 
the classical Molien series 
gives the linearly independent homogeneous invariant polynomials of $G$.
Later, R.P.~Stanley~\cite{Stanley1979} introduce the 
notion of \emph{bivariate} \emph{Molien series} that computes invariant polynomials 
of $G$ by their homogeneous degrees in $w,z$ and $x,y$.
R.P.~Stanley~\cite{Stanley1979} defined
the bivariate Molien series as follows:
\[
	f(u,v)
	:=
	\frac{1}{|G|}
	\sum_{g\in G}
	\frac{1}{\det{(1-ug)}\det{(1-vg)}}.
\]
A.~Bonnecaze et al.~\cite{BMS1999} showed that
the bivariate Molien series plays an important role to describe the relation
between the Jacobi polynomials of codes and designs such as group divisible designs,
packing (resp. covering) designs.
To compute all the invariant polynomials of~$G$ explicitly, it is convenient
to classify the invariants by their degrees.
We denote the homogeneous part of degree $d$ of $f(u,v)$ by $f[d]$.

In the following examples, we study two types of codes over $\FF_{q}$; 
Type~III and Type~IV, 
that hold $t$-designs with parameters $t$-$(v,k,\lambda)$, 
and we would like to give an upper (resp. lower) bound of $D_{\lambda}(v,k,t)$ 
(resp. $C_{\lambda}(v,k,t)$) of a packing (resp. covering) design 
corresponding to the parameters.
To do so, firstly, we compute the homogeneous part $f[d]$ of $f(u,v)$ 
corresponding to a code of length~$d$.
The coefficients of $f[d]$ determine the number of polynomials that are needed to generate the space of Jacobi polynomials corresponding to the reference sets with a particular cardinality.
The number of those Jacobi polynomials determines the number of $\lambda$'s of $t$-$(n,k,\lambda_1^{a_1},\dots,\lambda_N^{a_N})$.
Finally, the coefficient of the term $x^{n-k}y^k$ in the weight enumerator of the code 
obtains the upper (resp. lower) bound of $D_{\lambda}(v,k,t)$ (resp. $C_{\lambda}(v,k,t)$).
Note that a packing (resp. covering) design is a simple design.


\subsection{Type III codes}

The MacWilliams identity and the modulo~$3$ congruence condition yield that 
the weight enumerator of a Type III code remains invariant under the action of 
group $G_3$ of order $48$ which is generated by the following two matrices:
\[
	\frac{1}{\sqrt{3}}
	\begin{bmatrix}
		1&2\\
		1&-1
	\end{bmatrix}
	\text{ and }
	\begin{bmatrix}
		1&0\\
		0&e^{2\pi i/3}
	\end{bmatrix}.
\]
For the case of the group $G_3$, we get from the Magma computations 
that the denominator of $f(u,v)$ is in the form of $d(u)d(v)$, where
\[
d(u)
=
(u-1)^2(u+1)^2(u^2+1)^2(u^2-u+1)(u^2+u+1)(u^4-u^2+1).
\]

\begin{ex}[length $4$]
	Let $C_4^{\text{III}}$ be a ternary self-dual code of length $4$ in \cite{HM}. 
	Then
	\[
	f[4]=u^4+u^3v+u^2v^2+uv^3+v^4.
	\]
	Since $C_4^{\text{III}}$ holds $3$-design, we assume that
	$|T|=1,2,3$. Then
	\begin{align*}
		J_{C_4^{\text{III}},1}&=\frac{1}{4}AW_{C_4^{\text{III}}}(x,y)\\
		&=w(x^3+2y^3)+6zxy^2,\\
		J_{C_4^{\text{III}},2}&=\frac{1}{4\cdot3}A^2W_{C_4^{\text{III}}}(x,y)\\
		&=w^2x^2+4wzy^2+4z^2xy,\\
		J_{C_4^{\text{III}},3}&=\frac{1}{4\cdot3\cdot2}A^3W_{C_4^{\text{III}}}(x,y)\\
		&=w^3x+6wz^2y+2z^3x.
	\end{align*}

By dividing the coefficient of the term $z^ty^{3-t}$ ($t=1,2,3$) in the Jacobi polynomial by $2$, we obtain the values of $\lambda$.
Since the coefficient of the term $u^{4-t}v^t$ in $f[d]$  is $1$, 
we obtain the group divisible design $t$-$(4,3,\lambda)$. 
Then the maximum number of blocks of $t$-$(4,3,\lambda)$ design is 4, and the minimum number of blocks of $t$-$(4,3,\lambda)$ design is 4.
Therefore, $D_{\lambda}(4,3,t)\leq 4\leq C_{\lambda}(4,3,t)$.
\end{ex}

\begin{ex}[length $8$]
	Let $C_8^{\text{III}}$ be a ternary self-dual code of length $8$ in \cite{HM}. 
	Then
	\[
	f[8]=u^8+u^7v+2u^6v^2+2u^5v^3+2u^4v^4+2u^3v^5+2u^2v^6+uv^7+v^8.
	\]
	Since $C_8^{\text{III}}$ holds $1$-design, we assume that
	$|T|=1$. Then
	\[
	J_{C_8^{\text{III}},1}=\frac{1}{8}AW_{C_8^{\text{III}}}(x,y)=w(x^7+10x^4y^3+16xy^6)+z(6x^5y^2+48x^2y^5).
	\]
	The space of Jacobi polynomials $J_{C_8^{\text{III}},T}$ with $|T|=2$ may be generated by the two polynomials
	\begin{align*}
		J^1_{C_8^{\text{III}},2}&=w^2(x^6+8x^3y^3)+wz(4x^4y^2+32xy^5)+z^2(4x^5y+32x^2y^4),\\ 
		J^2_{C_8^{\text{III}},2}&=w^2(4x^3y^3+4y^6)+wz(12x^4y^2+24xy^5)+36z^2x^2y^4. 
	\end{align*}
	Combining these two equations, we obtain $2$-designs with parameters
	\begin{align*}
		2&\text{-}(8,3,(2^{12},0^{16})),\\
		2&\text{-}(8,6,(8^{12},9^{16})).
	\end{align*}
		Since $k=3\ell$ ($1\leq \ell\leq 2$), dividing the coefficient of 
		the term $z^ty^{k-t}$ ($t=2,3$) in the Jacobi polynomials by $2^\ell$, 
		we obtain the values of $\lambda_1,\lambda_2$.
	Dividing the coefficient of the term $x^{8-k}y^k$ in the weight enumerator of the code by $2^\ell$, we obtain an upper (resp. lower) bound of $D_{\lambda}(8,k,t)$ (resp. $C_{\lambda}(8,k,t)$).
	\begin{align*}
		D_2(8,3,2)&\leq 8\leq C_0(8,3,2),\\
		D_9(8,6,2)&\leq 16\leq C_8(8,6,2).
	\end{align*}
	The space of Jacobi polynomials $J_{C_8^{\text{III}},T}$ with $|T|=3$ may be generated by the two polynomials
	\begin{align*}
		J^1_{C_8^{\text{III}},3}&=w^3(x^5+8x^2y^3)+wz^2(6x^4y+48xy^4)+z^3(2x^5+16x^2y^3),\\
		J^2_{C_8^{\text{III}},3}&=w^3(x^5+2x^2y^3)+w^2z(10x^3y^2+8y^5)\\
		&~+wz^2(4x^4y+32xy^4)+24z^3x^2y^3,
	\end{align*}
	which gives packing and covering designs
	\begin{align*}
		D_1(8,3,3)&\leq 8\leq C_0(8,3,3),\\
		D_6(8,6,3)&\leq 16\leq C_4(8,6,3).
	\end{align*}
\end{ex}

\begin{ex}[length $12$]
	Let $C_{12}^{\text{III}}$ be the first ternary self-dual code of length $12$ in \cite{HM}. 
	\[
	f[12]=2u^{12}+2u^{11}v+3u^{10}v^2+4u^9v^3+4u^8v^4+4u^7v^5+5u^6v^6+\cdots.
	\]
	Since $C_{12}^{\text{III}}$ holds $5$-design, we observe that
	\begin{align*}
		J_{C_{12}^{\text{III}},1}&=\frac{1}{12}AW_{C_{12}^{\text{III}}}(x,y)\\
		&=w(x^{11}+132x^5y^6+110x^2y^9)+z(132x^6y^5+330x^3y^8+24y^{11}),\\
		J_{C_{12}^{\text{III}},2}&=\frac{1}{12\cdot11}A^2W_{C_{12}^{\text{III}}}(x,y)\\
		&=w^2(x^{10}+60x^4y^6+20xy^9)+2wz(72x^5y^5+90x^2y^8)\\
		&~+z^2(60x^6y^4+240x^3y^7+24y^{10}),\\
		J_{C_{12}^{\text{III}},3}&=\frac{1}{12\cdot11\cdot10}A^3W_{C_{12}^{\text{III}}}(x,y)\\
		&=w^3(x^9+24x^3y^6+2y^9)+w^2z(108x^4y^5+54xy^8)\\
		&~+wz^2(108x^5y^4+216x^2y^7)\\
		&+z^3(24x^6y^3+168x^3y^6+24y^9),\\
		J_{C_{12}^{\text{III}},4}&=\frac{1}{12\cdot11\cdot10\cdot9}A^4W_{C_{12}^{\text{III}}}(x,y)\\
		&=w^4(x^8+8x^2y^6)+w^3z(64x^3y^5+8y^8)+w^2z^2(120x^4y^4+96xy^7)\\
		&+wz^3(64x^5y^3+224x^2y^6)+z^4(8x^6y^2+112x^3y^5+24y^8),\\
		J_{C_{12}^{\text{III}},5}&=\frac{1}{12\cdot11\cdot10\cdot9\cdot8}A^5w_{C_{12}^{\text{III}}}(x,y)\\
		&=w^5(x^7+2xy^6)+30w^4zx^2y^5+w^3z^2(100x^3y^4+20y^7)\\
		&+wz^4(30x^5y^2+210x^2y^5)+w^2z^3(100x^4y^3+140xy^6)\\
		&~+z^5(2x^6y+70x^3y^4+24y^7).
	\end{align*}
	The space of Jacobi polynomials $J_{C_{12}^{\text{III}},T}$ with $|T|=6$ is generated by the two polynomials
	\begin{align*}
		J^1_{C_{12}^{\text{III}},6}&=w^6(x^6+2y^6)+90w^4z^2x^2y^4+w^3z^3(80x^3y^3+40y^6)\\
		&~+w^2z^4(90x^4y^2+180xy^5)+180wz^5x^2y^4\\
		&~+z^6(2x^6+40x^3y^3+24y^6),\\
		J^2_{C_{12}^{\text{III}},6}&=w^6x^6+12w^5zxy^5+60w^4z^2x^2y^4+w^3z^3(120x^3y^3+40y^6)\\
		&~+w^2z^4(60x^4y^2+180xy^5)+wz^5(12x^5y+180x^2y^4)\\
		&~+z^6(40x^3y^3+24y^6),
	\end{align*}
	which gives packing and covering designs
	\[
	D_1(12,6,6)\leq 132\leq C_0(12,6,6).
	\]
\end{ex}

\begin{ex}[length $16$]
	Let $C_{16}^{\text{III}}$ be the seventh ternary self-dual code of length $16$ in \cite{HM}.  
	\begin{align*}
	f[16]&=2u^{16}+3u^{15}v+4u^{14}v^2+5u^{13}v^3+6u^{12}v^4+6u^{11}v^5+7u^{10}v^6\\
	&~+7u^9v^7+7u^8v^8+\cdots.
	\end{align*}
	Observe that
	\begin{align*}
		&J_{C_{16}^{\text{III}},1}=\frac{1}{16}AW_{C_{16}^{\text{III}}}(x,y)\\
		&=w(x^{15}+140x^9y^6+1190x^6y^9+840x^3y^{12}+16y^{15})\\
		&~+z(84x^{10}y^5+1530x^7y^8+2520x^4y^{11}+240xy^{14}),\\
		&J_{C_{16}^{\text{III}},2}=\frac{1}{16\cdot15}A^2W_{C_{16}^{\text{III}}}(x,y)\\
		&=w^2(x^{14}+84x^8y^6+476x^5y^9+168x^2y^{12})\\
		&~+wz(112x^9y^5+1428x^6y^8+1344x^3y^{11}+32y^{14})\\
		&~+z^2(28x^{10}y^4+816x^7y^7+1848x^4y^{10}+224xy^{13}),\\
		&J_{C_{16}^{\text{III}},3}=\frac{1}{16\cdot15\cdot14}A^3W_{C_{16}^{\text{III}}}(x,y)\\
		&=w^3(x^{13}+48x^7y^6+170x^4y^9+24xy^{12})\\
		&~+w^2z(108x^8y^5+918x^5y^8+432x^2y^{11})\\
		&~+wz^2(60x^9y^4+1224x^6y^7+1584x^3y^{10}+48y^{13})\\
		&~+z^3(8x^{10}y^3+408x^7y^6+1320x^4y^9+208xy^{12}).
	\end{align*}
	The space of Jacobi polynomials $J_{C_{16}^{\text{III}},T}$ with $|T|=4$ may be generated by the four polynomials
	\begin{align*}
		&J^1_{C_{16}^{\text{III}},4}=w^4(x^{12}+32x^6y^6+40x^3y^9+8y^{12})\\
		&~+w^3z(64x^7y^5+520x^4y^8+64xy^{11})\\
		&~+w^2z^2(120x^8y^4+1056x^5y^7+768x^2y^{10})\\
		&~+wz^3(928x^6y^6+1600x^3y^9+64y^{12})\\
		&~+z^4(8x^{10}y^2+176x^7y^5+920x^4y^8+192xy^{11}), \\ 
		&J^2_{C_{16}^{\text{III}},4}=w^4(x^{12}+24x^6y^6+56x^3y^9)\\
		&~+w^3z(96x^7y^5+456x^4y^8+96xy^{11})\\
		&~+w^2z^2(72x^8y^4+1152x^5y^7+720x^2y^{10})\\
		&~+wz^3(32x^9y^3+864x^6y^6+1632x^3y^9+64y^{12})\\
		&~+z^4(192x^7y^5+912x^4y^8+192xy^{11}), \\ 
		&J^3_{C_{16}^{\text{III}},4}=w^4(x^{12}+26x^6y^6+52x^3y^9+2y^{12})\\
		&~+w^3z(88x^7y^5+472x^4y^8+88xy^{11})\\
		&~+w^2z^2(84x^8y^4+1128x^5y^7+732x^2y^{10})\\
		&~+wz^3(24x^9y^3+880x^6y^6+1624x^3y^9+64y^{12})\\
		&~+z^4(2x^{10}y^2+188x^7y^5+914x^4y^8+192xy^{11}), \\
		&J^4_{C_{16}^{\text{III}},4}=w^4(x^{12}+28x^6y^6+48x^3y^9+4y^{12})\\
		&~+w^3z(80x^7y^5+488x^4y^8+80xy^{11})\\
		&~+w^2z^2(72x^8y^4+1104x^5y^7+744x^2y^{10})\\
		&~+wz^3(16x^9y^3+896x^6y^6+1616x^3y^9+64y^{12})\\
		&~+z^4(4x^{10}y^2+184x^7y^5+916x^4y^8+192xy^{11}), 
	\end{align*}
	which gives packing and covering designs
	\begin{align*}
		D_4(16,6,4)&\leq 112\leq C_0(16,6,4),\\
		D_{96}(16,9,4)&\leq 1360\leq C_{88}(16,9,4),\\
		D_{343}(16,12,4)&\leq 1260\leq C_{342}(16,12,4).
	\end{align*}

\end{ex}

\begin{ex}[length $20$]
	Let $C_{20}^{\text{III}}$ be the $19$th ternary self-dual code of length $20$ in \cite{HM}. 
	\begin{align*}
		f[20]&=2u^{20}+3u^{19}v+5u^{18}v^2+6u^{17}v^3+7u^{16}v^4+8u^{15}v^5+9u^{14}v^6\\
		&~+9u^{13}v^7+10u^{12}v^8+10u^{11}v^9+10u^{10}v^{10}+\cdots.
	\end{align*}
	Observe that
	\begin{align*}
		&J_{C_{20}^{\text{III}},1}=\frac{1}{20}AW_{C_{20}^{\text{III}}}(x,y)\\
		&=w(x^{19}+84x^{13}y^6+2398x^{10}y^9+10512x^7y^{12}+6432x^4y^{15}+256xy^{18})\\
		&~+z(36x^{14}y^5+1962x^{11}y^8+15768x^8y^{11}+19296x^5y^{14}+2304x^2y^{17}).
	\end{align*}
	The space of Jacobi polynomials $J_{C_{20}^{\text{III}},T}$ with $|T|=2$ may be generated by the two polynomials
	\begin{align*}
		&J^1_{C_{20}^{\text{III}},2}=w^2(x^{18}+48x^{12}y^6+1300x^9y^9+3816x^6y^{12}+1392x^3y^{15}+4y^{18})\\
		&~+wz(72x^{13}y^5+2196x^{10}y^8+13392x^7y^{11}+10080x^4y^{14}+504xy^{17})\\
		&~+z^2(864x^{11}y^7+9072x^8y^{10}+14256x^5y^{13}+2052x^2y^{16}),\\ 
		&J^2_{C_{20}^{\text{III}},2}=w^2(x^{18}+68x^{12}y^6+1220x^9y^9+3936x^6y^{12}+1312x^3y^{15}+24y^{18})\\
		&~+wz(32x^{13}y^5+2356x^{10}y^8+13152x^7y^{11}+10240x^4y^{14}+464xy^{17})\\
		&~+z^2(20x^{14}y^4+784x^{11}y^7+9192x^8y^{10}+14176x^5y^{13}+2072x^2y^{16}), 
	\end{align*}
	which gives packing and covering designs
	\begin{align*}
		D_{10}(20,6,2)&\leq 60\leq C_0(20,6,2),\\
		D_{432}(20,9,2)&\leq 2180\leq C_{392}(20,9,2),\\
		D_{4296}(20,12,2)&\leq 12240\leq C_{4212}(20,12,2),\\
		D_{6444}(20,15,2)&\leq 11544\leq C_{6308}(20,15,2).
	\end{align*}
\end{ex}

\subsection{Type IV codes}

It is well-known (see~\cite{NRS}) that 
the weight enumerator of a Type IV code remains invariant under 
the action of group $G_4$ of order $12$ 
which is generated by the following two matrices:
\[
	\frac{1}{2}
	\begin{bmatrix}
		1&3\\
		1&-1
	\end{bmatrix}
	\text{ and }
	\begin{bmatrix}
		1&0\\
		0&-1
	\end{bmatrix},
\]
which corresponds to the MacWilliams identity and the modulo~$2$ congruence condition,
respectively.
In particular, for the case of the group~$G_4$ a Magma computation gives the 
denominator $d(u)d(v)$ of $f(u,v)$, where
\[
	d(u)
	=
	(1-u+u^2) (1+u+u^2) (1+2u^6+3u^{12}+4u^{18}+5u^{24}+6u^{30}+7u^{36}).
\]


\begin{ex}[length $2$]
	Let $C_2^{\text{IV}}$ be a Hermitian self-dual code over $\FF_4$ of length $2$ in \cite{HM}. 
	Then
	\[
	f[2]=u^2+uv+v^2
	\]
	If $|T|=1$, we have
	\begin{align*}
		J_{C_2^{\text{IV}},1}&=\frac{1}{2}AW_{C_2^{\text{IV}}}(x,y)\\
		&=wx+3zy.
	\end{align*}
\end{ex}

\begin{ex}[length $4$]
	Let $C_4^{\text{IV}}$ be a Hermitian self-dual code over $\FF_4$ of length $4$ in \cite{HM}. 
	\[
	f[4]=u^4+u^3v+2u^2v^2+uv^3+v^4.
	\]
	Observe that
	\[
	J_{C_4^{\text{IV}},1}=\frac{1}{4}AW_{C_4^{\text{IV}}}(x,y)=wx^3+3wxy^2+3zx^2y+9zy^3.
	\]
	The space of Jacobi polynomials $J_{C_4^{\text{IV}},T}$ with $|T|=2$ is generated by the two polynomials
	\begin{align*}
		J^1_{C_4^{\text{IV}},2}&=w^2x^2+6wzxy+9z^2y^2,\\
		J^2_{C_4^{\text{IV}},2}&=w^2x^2+3w^2y^2+3z^2x^2+9z^2y^2.
	\end{align*}
	Combining these two equations we obtain $2$-designs with parameters
	\[
	2\text{-}(4,2,0^{4},1^{2}).
	\]
Since $k=2$, dividing the coefficient of the term $z^4y^2$ in the Jacobi polynomials by $2$, we obtain the values of $\lambda_1,\lambda_2$.
	This gives packing and covering designs 
	\[
	D_1(4,2,2)\leq 2\leq C_0(4,2,2).
	\]
\end{ex}

\begin{ex}[length $6$]
	Let $C_6^{\text{IV}}$ be the first Hermitian self-dual code over $\FF_4$ of length $6$ in \cite{HM}. 
	\[
	f[6]=2u^6+2u^5v+3u^4v^2+3u^3v^3+3u^2v^4+2uv^5+2v^6.
	\]
	Observe that
	\[
		J_{C_6^{\text{IV}},1}
		=
		\frac{1}{6}
		AW_{C_6^{\text{IV}}}(x,y)
		=
		wx^5+6wx^3y^2+9wxy^4+3zx^4y+18zx^2y^3+27y^5.
	\]
	The space of Jacobi polynomials $J_{C_6^{\text{IV}},T}$ with $|T|=2$ may be generated by the two polynomials
	\begin{align*}
		J^1_{C_6^{\text{IV}},2}&=w^2(x^4+6x^2y^2+9y^4)+z^2(3x^4+18x^2y^2+27y^4),\\
		J^2_{C_6^{\text{IV}},2}&=w^2(x^4+3x^2y^2)+wz(6x^3y+18xy^3)+z^2(9x^2y^2+27y^4).
	\end{align*}
	Combining these two equations we obtain $2$-designs with parameters
	\begin{align*}
		2&\text{-}(6,2,0^{12},1^{3}),\\
		2&\text{-}(6,4,1^{12},2^{3}).
	\end{align*}
	Since $k=2\ell$ ($1\leq \ell\leq 2$), by dividing the coefficient of the term $z^2y^{k-2}$ in the Jacobi polynomials by $3^\ell$, we obtain the values of $\lambda_1,\lambda_2$.
	This gives packing and covering designs 
	\begin{align*}
		D_1(6,2,2)&\leq 3\leq C_0(6,2,2),\\
		D_2(6,4,2)&\leq 3\leq C_1(6,4,2).
	\end{align*}
\end{ex}

\begin{ex}[length $8$]
	Let $C_8^{\text{IV}}$ be the third Hermitian self-dual code over $\FF_4$ of length $8$ in \cite{HM}. 
	\[
	f[8]=2u^8+3u^7v+4u^6v^2+4u^5v^3+5u^4v^4+4u^3v^5+4u^2v^6+3uv^7+2v^8.
	\]
	Observe that
	\begin{align*}
		J_{C_8^{\text{IV}},1}&=\frac{1}{8}AW_{C_8^{\text{IV}}}(x,y)\\
		&=w(x^7+21x^3y^4+42xy^6)+z(21x^4y^3+126x^2y^5+45y^7),\\
		J_{C_8^{\text{IV}},2}&=\frac{1}{8\cdot7}A^2W_{C_8^{\text{IV}}}(x,y)\\
		&=w^2(x^6+9x^2y^4+6y^6)+wz(24x^3y^3+72xy^5)\\
		&~+z^2(9x^4y^2+90x^2y^4+45y^6),\\
		J_{C_8^{\text{IV}},3}&=\frac{1}{8\cdot7\cdot6}A^3W_{C_8^{\text{IV}}}(x,y)\\
		&=w^3(x^5+3xy^4)+w^2z(18x^2y^3+18y^5)\\
		&~+wz^2(18x^3y^2+90xy^4)+z^3(3x^4y+60x^2y^3+45y^5).\\
	\end{align*}
	The space of Jacobi polynomials $J_{C_8^{\text{IV}},T}$ with $|T|=4$ may be generated by the two polynomials
	\begin{align*}
		J^1_{C_8^{\text{IV}},4}&=w^4(x^4+3y^4)+36w^2z^2(x^2y^2+36y^4)\\
		&~+96wz^3xy^3+z^4(3x^4+36x^2y^2+45y^4),\\
		J^2_{C_8^{\text{IV}},4}&=w^4x^4+12w^3zxy^3+w^2z^2(18x^2y^2+36y^4)\\
		&~+wz^3(12x^3y+96xy^3)+z^4(36x^2y^2+45y^4).
	\end{align*}
	Combining these two equations we obtain $4$-designs with parameters
	\[
	4\text{-}(8,4,0^{56},1^{14}).
	\]
	This gives packing and covering designs
	\[
	D_1(8,4,4)\leq 14\leq C_0(8,4,4).
	\]
\end{ex}


\begin{ex}[length $12$]
	Let $C_{12}^{\text{IV}}$ be the seventh Hermitian self-dual code over $\FF_4$ of length $12$ in \cite{HM}. 
	\begin{align*}
	f[12]&=3u^{12}+4u^11v+6u^{10}v^2+7u^9v^3+8u^8v^4+8u^7v^5+9u^6v^6+8u^5v^7\\
	&~+8u^4v^8+7u^3v^9+6u^2v^{10}+4uv^{11}+3v^{12}.
	\end{align*}
	Observe that
	\begin{align*}
		J_{C_{12}^{\text{IV}},1}&=\frac{1}{12}AW_{C_{12}^{\text{IV}}}(x,y)\\
		&=w(x^{11}+30x^7y^4+108x^5y^6+585x^3y^8+300xy^{10})\\
		&~+z(15x^8y^3+108x^6y^5+1170x^4y^7+279y^{11}).
	\end{align*}
	The space of Jacobi polynomials $J_{C_{12}^{\text{IV}},T}$ with $|T|=2$ may be generated by the two polynomials
	\begin{align*}
		J^1_{C_{12}^{\text{IV}},2}&=w^2(x^{10}+30x^6y^4+60x^4y^6+105x^2y^8+60y^{10})\\
		&~+wz(96x^5y^5+960x^3y^7+480xy^9)\\
		&~+z^2(15x^8y^2+60x^6y^4+690x^4y^6+1260x^2y^8+279y^{10}),\\ 
		J^2_{C_{12}^{\text{IV}},2}&=w^2(x^{10}+18x^6y^4+48x^4y^6+165x^2y^8+24y^{10})\\
		&~+wz(24x^7y^3+120x^5y^5+840x^3y^7+552xy^9)\\
		&~+z^2(3x^8y^2+48x^6y^4+750x^4y^6+1224x^2y^8+279y^{10}), 
	\end{align*}
	which gives packing and covering designs
	\begin{align*}
		D_5(12,4,2)&\leq 15 \leq C_1(12,4,2),\\
		D_{12}(12,6,2)&\leq 52 \leq C_{10}(12,6,2),\\
		D_{110}(12,8,2)&\leq 255 \leq C_{90}(12,8,2).
	\end{align*}
\end{ex}

\begin{ex}[length $14$]
	Let $C_{14}^{\text{IV}}$ be the first Hermitian self-dual code over $\FF_4$ of length $14$ in \cite{HM}. 
	\begin{align*}
		f[14]&=3u^{14}+5u^{13}v+7u^{12}v^2+8u^{11}v^3+10u^{10}v^4+10u^9v^5+11u^8v^6\\
		&~+11u^7v^7+11u^6v^8+10u^5v^9+10u^4v^{10}+8u^3v^{11}+7u^2v^{12}\\
		&~+5uv^{13}+3v^{14}.
	\end{align*}
	Observe that
	\begin{align*}
		J_{C_{14}^{\text{IV}},1}&=\frac{1}{14}AW_{C_{14}^{\text{IV}}}(x,y)\\
		&=w(x^{13}+18x^{11}y^2+135x^9y^4+540x^7y^6+1215x^5y^8+1458x^3y^{10}\\
		&~+729xy^{12})+z(3x^{12}y+54x^{10}y^3+405x^8y^5+1620x^6y^7+3645x^4y^9\\
		&~+4374x^2y^{11}+2187y^{13}).
	\end{align*}
	The space of Jacobi polynomials $J_{C_{14}^{\text{IV}},T}$ with $|T|=2$ may be generated by the two polynomials
	\begin{align*}
		J^1_{C_{14}^{\text{IV}},2}&=w^2(x^{12}+18x^{10}y^2+135x^8y^4+540x^6y^6+1215x^4y^8+1458x^2y^{10}\\
		&~+729y^{12})+z^2(3x^{12}+54x^{10}y^2+405x^8y^4+1620x^6y^6+3645x^4y^8\\
		&~+4374x^2y^{10}+2187y^{12}),\\ 
		J^2_{C_{14}^{\text{IV}},2}&=w^2(x^{12}+15x^{10}y^2+90x^8y^4+270x^6y^6+405x^4y^8+243x^2y^{10})\\
		&~+wz(6x^{11}y+90x^9y^3+540x^7y^5+1620x^5y^7+2430x^3y^9\\
		&~+1458xy^{11})+z^2(9x^{10}y^2+135x^8y^4+810x^6y^6+2430x^4y^8\\
		&~+3645x^2y^{10}+2187y^{12}), 
	\end{align*}
	which gives packing and covering designs
	\begin{align*}
		D_1(14,2,2)&\leq 7 \leq C_0(14,2,2),\\
		D_6(14,4,2)&\leq 21 \leq C_1(14,4,2),\\
		D_{15}(14,6,2)&\leq 35 \leq C_5(14,6,2),\\
		D_{20}(14,8,2)&\leq 35 \leq C_{10}(14,8,2),\\
		D_{15}(14,10,2)&\leq 21 \leq C_{10}(14,10,2),\\
		D_6(14,12,2)&\leq 7 \leq C_5(14,12,2).
	\end{align*}
\end{ex}

\begin{ex}[length $16$]
	Let $C_{16}^{\text{IV}}$ be the $35$th Hermitian self-dual code over $\FF_4$ of length $16$ in \cite{HM}. 
	\begin{align*}
		f[16]&=3u^{16}+5u^{15}v+8u^{14}v^2+9u^{13}v^3+11u^{12}v^4+12u^{11}v^5+13u^{10}v^6\\
		&~+13u^9v^7+14u^8v^8+13u^7v^9+13u^6v^{10}+12u^5v^{11}+11u^4v^{12}\\
		&~+9u^3v^{13}+8u^2v^{14}+5uv^{15}+3v^{16}.
	\end{align*}
	Observe that
	\begin{align*}
		J_{C_{16}^{\text{IV}},1}&=\frac{1}{16}AW_{C_{16}^{\text{IV}}}(x,y)\\
		&=w(x^{15}+21x^{13}y^2+189x^{11}y^4+945x^9y^6+2835x^7y^8+5103x^5y^{10}\\
		&~+5103x^3y^{12}+2187xy^{14})+z(3x^{14}y+63x^{12}y^3+567x^{10}y^5\\
		&~+2835x^8y^7+8505x^6y^9+15309x^4y^{11}+15309x^2y^{13}+6561y^{15}).
	\end{align*}
	The space of Jacobi polynomials $J_{C_{16}^{\text{IV}},T}$ with $|T|=2$ is generated by the 
	following two polynomials:
	\begin{align*}
		J^1_{C_{16}^{\text{IV}},2}&=w^2(x^{14}+21x^{12}y^2+189x^{10}y^4+945x^8y^6+2835x^6y^8+5103x^4y^{10}\\
		&~+5103x^2y^{12}+2187y^{14})+z^2(3x^{14}+63x^{12}y^2+567x^{10}y^4\\
		&~+2835x^8y^6+8505x^6y^8+15309x^4y^{10}+15309x^2y^{12}+6561y^{14}),\\ 
		J^2_{C_{16}^{\text{IV}},2}&=w^2(x^{14}+18x^{12}y^2+135x^{10}y^4+540x^8y^6+1215x^6y^8+1458x^4y^{10}\\
		&~+729x^2y^{12})+wz(6x^{13}y+108x^{11}y^3+810x^9y^5+3240x^7y^7\\
		&~+7290x^5y^9+8748x^3y^{11}+4374xy^{13})+z^2(9x^{12}y^2+162x^{10}y^4\\
		&~+1215x^8y^6+4860x^6y^8+10935x^4y^{10}+13122x^2y^{12}+6561y^{14}), 
	\end{align*}
	which gives packing and covering designs
	\begin{align*}
		D_1(16,2,2)&\leq 8 \leq C_0(16,2,2),\\
		D_7(16,4,2)&\leq 28 \leq C_1(16,4,2),\\
		D_{21}(16,6,2)&\leq 56 \leq C_6(16,6,2),\\
		D_{35}(16,8,2)&\leq 70 \leq C_{15}(16,8,2),\\
		D_{35}(16,10,2)&\leq 56 \leq C_{20}(16,10,2),\\
		D_{21}(16,12,2)&\leq 28 \leq C_{15}(16,12,2),\\
		D_7(16,14,2)&\leq 8 \leq C_6(16,14,2).
	\end{align*}
\end{ex}

\begin{ex}[length $18$]
	Let $C_{18}^{\text{IV}}$ be the $225$th Hermitian self-dual code over $\FF_4$ of length $18$ in \cite{HM}. 
	\begin{align*}
		f[18]&=4u^{18}+6u^{17}v+9u^{16}v^2+11u^{15}v^3+13u^{14}v^4+14u^{13}v^5+16u^{12}v^6\\		
		&~+16u^{11}v^7+17u^{10}v^8+17u^9v^9+17u^8v^{10}+16u^7v^{11}+16u^6v^{12}\\
		&~+14u^5v^{13}+13u^4v^{14}+11u^3v^{15}+9u^2v^{16}+6uv^{17}+4v^{18}.
	\end{align*}
	Observe that
	\begin{align*}
		J_{C_{18}^{\text{IV}},1}&=\frac{1}{18}AW_{C_{18}^{\text{IV}}}(x,y)\\
		&=w(x^{17}+24x^{15}y^2+252x^{13}y^4+1512x^{11}y^6+5670x^9y^8\\
		&~+13608x^7y^{10}+20412x^5y^{12}+17496x^3y^{14}+6561xy^{16})\\
		&~+z(3x^{16}y+72x^{14}y^3+756x^{12}y^5+4536x^{10}y^7+17010x^8y^9\\
		&~+40824x^6y^{11}+61236x^4y^{13}+52488x^2y^{15}+19683y^{17}).
	\end{align*}
	The space of Jacobi polynomials $J_{C_{18}^{\text{IV}},T}$ with $|T|=2$ may be generated by the two polynomials
	\begin{align*}
		J^1_{C_{18}^{\text{IV}},2}&=w^2(24x^{14}y^2+252x^{12}y^4+1512x^{10}y^6+5670x^8y^8+13608x^6y^{10}\\
		&~+20412x^4y^{12}+17496x^2y^{14}+6561y^{16})+z^2(3x^{16}+72x^{14}y^2\\
		&~+756x^{12}y^4+4536x^{10}y^6+17010x^8y^8+40824x^6y^{10}\\
		&~+61236x^4y^{12}+52488x^2y^{14}),\\ 
		J^2_{C_{18}^{\text{IV}},2}&=w^2(21x^{14}y^2+189x^{12}y^4+945x^{10}y^6+2835x^8y^8+5103x^6y^{10}\\
		&~+5103x^4y^{12}+2187x^2y^{14})+wz(6x^{15}y+126x^{13}y^3+1134x^{11}y^5\\
		&~+5670x^9y^7+17010x^7y^9+x^5y^{11}+30618x^3y^{13}+13122xy^{15})\\
		&~+z^2(9x^{14}y^2+189x^{12}y^4+1701x^{10}y^6+8505x^8y^8+25515x^6y^{10}\\
		&~+45927x^4y^{12}+45927x^2y^{14}), 
	\end{align*}
	which gives packing and covering designs
	\begin{align*}
		D_1(18,2,2)&\leq 9 \leq C_0(18,2,2),\\
		D_8(18,4,2)&\leq 36 \leq C_1(18,4,2),\\
		D_{28}(18,6,2)&\leq 84 \leq C_7(18,6,2),\\
		D_{56}(18,8,2)&\leq 126 \leq C_{21}(18,8,2),\\
		D_{70}(18,10,2)&\leq 126 \leq C_{35}(18,10,2),\\
		D_{56}(18,12,2)&\leq 84 \leq C_{35}(18,12,2),\\
		D_{28}(18,14,2)&\leq 36 \leq C_{21}(18,14,2),\\
		D_8(18,16,2)&\leq 9 \leq C_7(18,16,2).
	\end{align*}
\end{ex}

\section{Concluding remarks}\label{Sec:Conclusion}

The $g$-th Jacobi polynomials of a binary code were introduced in~\cite{HOO}
which were generalized in~\cite{CMO2022} to the case of a non-binary code. 
This rises a natural question: is there any possibility to give 
a generalization of Theorem~\ref{Thm:Main2} for higher genus cases? 
We shall answer this question in~\cite{CHMOxxxx}. 
The study of this paper will be continued in~\cite{CIMTxxxx} to the case colored $t$-design, the idea that was introduce in~\cite{BRS2000}.
Moreover, we shall give the generalizations of the results in~\cite{CMO2022} 
for the $g$-th Jacobi polynomial with multiple reference vectors in~\cite{CMOxxxx}.

\section*{Acknowledgements}

This work was supported by JSPS KAKENHI (22K03277).



\end{document}